\documentclass[11pt]{amsart}
\usepackage{amsmath,amssymb,amsfonts,amsthm}
\usepackage{varioref}
\usepackage[utf8x]{inputenc}
\usepackage[T1]{fontenc}
\usepackage[english]{babel}
\usepackage{color} 
 
\everymath{\displaystyle}

\usepackage[bottom=3.5cm]{geometry}

\newtheorem{thm}{Theorem}[section]

\newtheorem{prop}[thm]{Proposition}

\newtheorem{lem}[thm]{Lemma}
\theoremstyle{plain}

\theoremstyle{remark}
\newtheorem*{req}{Remark}

\newcommand{\e}{\ensuremath{\varepsilon}}
\def\R{\mathbb{R}}
\def\Z{\mathbb{Z}}
\def\N{\mathbb{N}}

\def\C{\hat{\mathcal{C}}}
\def\Cs{\mathcal{C}}

\title{On the Time Constant in a Dependent First Passage Percolation Model}

{
\author{Julie Scholler}
\address{Institut \'Elie Cartan Nancy (math{\'e}matiques)\\
Universit{\'e} Henri Poincar{\'e} Nancy 1\\
Campus Scientifique, BP 239 \\
54506 Vandoeuvre-l{\`e}s-Nancy  Cedex France\\}
\email{Julie.Scholler@iecn.u-nancy.fr}
}


\begin{document}
\renewcommand{\proofname}{Proof}

\begin{abstract} 

We pursue the study of a random coloring first passage percolation model introduced by Fontes and Newman.
We prove that the asymptotic shape of this first passage percolation model continuously depends on the law of the coloring.
The proof uses several couplings, particularly with  greedy lattice animals.
\end{abstract} 
\maketitle
\bigskip

\section{Introduction}

First passage percolation model was introduced by Hammersley and Welsh~\cite{HW} as a stochastic model for a porous media. To each edge $e$ of the simple cubic lattice, we assign a random non negative number $t(e)$, which is interpreted as the time to cross the edge in either direction. Classically the random variables $t(e)$ are independent identically distributed (i.i.d.). In this article, we study a specific dependent model introduced by Fontes and Newman~\cite{fetn}. To construct it, we begin with an independent identically distributed random coloring of $\Z^d$. Then the random variables $t(e)$ are defined as follows: if the endpoints of the edge $e$ are in different colors, then the random variable $t(e)$ is equal to 1, otherwise it is equal to $0$. $\mathbb{Z}^d$ is split into color clusters representing countries and thus one spends one unit of time to go through a border. As in the standard model, there exists a seminorm $\mu_p$ ($p$ refers to the coloring law) on $\mathbb{R}^d$ which governs the propagation speed.

We show that the seminorm on the unit ball is continuous with respect to the coloring law and that the asymptotic shape is also continuous with respect to the coloring law for the Hausdorff distance.

Up to now, the seminorm continuity has been studied only in the case of standard first passage percolation by Cox~\cite{cox} and later by Cox and Kesten~\cite{coxkesten}. We will not use the same method as in the previous articles. We will proceed by taking advantage of the possibility to construct an interesting coupling of two random colorings and of the relation of the model with greedy lattice animals.

\section{Model and results}

\label{model}

\subsection{Random colorings and the associated passage times} 
Let $\Z^d$ be the integer lattice of dimension $d \geq 2$. Let $(X_u)_{u \in \Z ^d}$ be a family of positive integer valued independent identically distributed random variables. If we consider positive integers to be colors, then the joint law of the family $(X_u)_{u \in \Z ^d}$ determines a random coloring of $\Z^d$. We endow $\Z^d$ with the set of edges $\mathbb{E}^d$ defined by 
$$\mathbb{E}^d~=~\left\lbrace  \{u,v\};  u \in \Z^d, \ v \in \Z^d, \ \|u-v\|_1 := \sum_{i=1}^d |u_i-v_i|=1 \right\rbrace.$$ 
Now random colorings give rise to a first passage percolation model by defining for each edge $\{u,v\}$ in $\mathbb{E}^d$, $\ t \big( \{ u,v\} \big):=\textbf{1}_{\{X_u \neq X_v \}}$, where $\textbf{1}_{A}$ is the indicator function of the set $A$.

If the coloring is regarded as representing a geographical map, then each individual color cluster represents a single country and one can travel instantaneously inside a country but one spends one unit of time to cross a border.

Let us note that, though the $X_u$, $u \in \Z^d$, are independent, the random variables $t(e)$, $e \in \mathbb{E}^d $, are not. This is the difference with the standard first passage percolation model. We refer to Kesten's Saint Flour Notes~\cite{kesten} for a introduction to the subject and to Kesten's survey~\cite{surveyK} or Howard's survey~\cite{howard} for more recent results.

We define $$\mathcal{P} := \left\lbrace (p_1,p_2, \ldots ) \in [0,1]^{\N^*}; \sum_{i=1}^{+ \infty} p_i =1 \right\rbrace.$$ Let us remark that the $L_1$-norm, the supremum norm and the pointwise convergence lead to the same topology on $\mathcal{P}$ (this can be proved with Scheffe's Lemma~\cite{billingsley}). Thus we will use indistinctly the $L_1$ norm ($\| \cdot \|_1$) or the supremum norm ($| \cdot |$) on $\mathcal{P}$. Let $p=(p_1,p_2, \ldots)$ be in $\mathcal{P}$, then the law of  $X_u$ is $\sum_{i=1}^{+ \infty} p_i \delta_{i}$.

\subsection{Notation and general definitions for first passage percolation}
A \textit{path} $\gamma$ of \textit{length} $k$ is a sequence of edges and vertices, $\gamma=(u_1,e_1, \ldots , e_k, u_{k+1})$, such that for each integer $j$ between $1$ and $k$, $u_j$ is a vertex of $\Z^d$ and  $e_j$ is the edge between $u_j$ and $u_{j+1}$, \textit{i.e.}~$e_j=\{u_j,u_{j+1}\}$. So the length of a path $\gamma$, denoted by $l(\gamma)$, is the number of its edges. The number of vertices of a path $\gamma$ is denoted by $|\gamma|$.
Sometimes we will only consider \textit{self-avoiding} paths for which the vertices are all distinct. We define the \textit{passage time} of the path $\gamma$ to be the sum of the passage times through the edges of $\gamma$: $$T(\gamma) :=\sum_{i=1}^k t(e_i).$$ If $u$ and $v$ are vertices of $\Z^d$, then $\Gamma(u,v)$ denotes the set of the paths from $u$ to $v$ and $T(u,v)$ the passage time from $u$ to $v$, \textit{i.e.} $T(u,v):= \inf \left\lbrace T(\gamma) \ ; \ \gamma \in \Gamma(u,v) \right\rbrace.$
If a path $\gamma$ from $u$ to $v$ satisfies $T(\gamma)=T(u,v)$, then $\gamma$ is called a \textit{route}. Let us note that, for first passage percolation on a random coloring, $T(u,v)$ is a minimum, that is, for any given pair of vertices, there is always a route.

We can extend the definition of passage time between two sites of $\Z^ d$ to points of $\mathbb{R}^d$ as follows: to each point $x$ of $\mathbb{R}^ d$, we assign the nearest site $x^{\star}$ of $\mathbb{Z}^d$ with deterministic rules in case of equality. We define $T(x,y)$ to be $T(x^{\star},y^{\star})$.
One of the main topics of interest in first passage percolation is the set $B(t)$ of points that can be reached from the origin by time $t$, \textit{i.e.} $$B(t)= \overline{\left\lbrace x \in \mathbb{R}^d; T(0,x) \leq t\right\rbrace} .$$

\subsection{The seminorm $\mu$}
First of all, let us note that the joint law of the passage times induced by an i.i.d. random coloring is translation invariant and ergodic. Moreover, the random variables $t(e)$ are bounded by $1$. In the following, $T_p$ denotes the passage time corresponding to the coloring law $p$. With this notation, we have the following result:
\begin{prop}
For each $p$ in $\mathcal{P}$ and for each $x$ in $\mathbb{R}^d$, there exists a constant $\mu_p(x)$ such that
$$\lim_{n \rightarrow + \infty} \frac{T_p(0,nx)}{n} = \mu_p(x) \text{ a.s.}$$

Furthermore, $\mu_p$ is a seminorm on $\mathbb{R}^d$.
\end{prop}

The standard basis of $\R^d$ is denoted by $( \e_1, \ldots , \e_d )$.
Let us remark that, when the joint law of the passage times is invariant under permutations of the coordinates, the constants $\mu(\e_i)$ are the same. This constant $\mu(\e_1)$ is known as the \textit{time constant}.

For $x$ in $\Z^d$, the convergence follows from Kingman's subadditive ergodic theorem~\cite{MR0254907} (we can also use Liggett's improved version~\cite{MR806224}). In addition, the limit is the lower bound and is also the lower bound of the expectation. The definition of the function $\mu_p$ can be extended to $ \mathbb{Q}^d$ by defining $\mu_p \left( \frac{m}{k} \right)= \frac{1}{k} \mu_p(m)$ for $m$ in $\Z^d$ and $k$ a non-zero integer. Since the function $\mu_p$ is 1-Lipschitz continuous on $ \mathbb{Q}^d$, it can be extended to $\mathbb{R}^d$. We can actually check that $\frac{T_p(0,nx)}{n}$ converges to $\mu_p(x)$ when $n$ goes to infinity. 

Our main result concerns the continuity with respect to the coloring law of this seminorm. To clearly state it, we denote by $\mathcal{B}(0,1)$ the closed unit ball for the euclidean norm and by $\| \cdot \|_{\infty}$ the infinity norm on $\mathcal{B}(0,1)$, \textit{i.e.} $\|f\|_{\infty} = \sup_{x \in \mathcal{B}(0,1)} | f(x) |$. Then our main result is:

\begin{thm}\label{uni}
 The function $\mu$ :
\begin{tabular}{ccc}
 $\mathcal{P}$ & $\rightarrow$ & $ \mathcal{C} \left( \mathcal{B}(0,1) , \R \right)$ \\
$p$ & $\mapsto$ & $\left( x \mapsto \mu_p(x)\right)$
\end{tabular} is continuous with respect to the infinity norm.
\end{thm}

\begin{req}\label{princ}
It follows that, for each $x$ in $\R^d$, the function
\begin{tabular}{ccc}
 $\mathcal{P}$ & $\rightarrow$ & $\R ^+$ \\
$p$ & $\mapsto$ & $\mu_p(x)$
\end{tabular}
is continuous.
\end{req}

\subsection{The asymptotic shape}

Results on the asymptotic shape in standard first passage percolation have been extended to the case of stationary ergodic passage times with a moment condition by Boivin~\cite{boivin}. However, in first passage percolation on an i.i.d. random coloring, the passage times are bounded. It is a simpler case which was remarked by Derriennic and indicated at the end of Kesten's Saint-Flour notes~\cite{kesten}.
Thus, in the case of first passage percolation on a random coloring, we have the following result:

\begin{prop}
Let $B_{\mu_p} = \{x \in \R^d ; \mu_p(x) \leq 1 \}$.
\begin{enumerate}
 \item If $\mu_p(\e_1)>0$, then the set $B_{\mu_p}$ is a deterministic convex compact set with nonempty interior and for any $\e>0$, we have $$(1- \e)B_{\mu_p} \subset \frac{B(t)}{t} \subset (1+ \e) B_{\mu_p} \text{ for $t$ large enough a.s.}$$
	\item If $\mu_p(\e_1)=0$, then $\mu_p\equiv 0$ (so $B_{\mu_p}=\R^d$) and for any compact set $K$, we have $K \subset \frac{B(t)}{t}$ for $t$ large enough a.s.
\end{enumerate}
\end{prop}

The set $B_{\mu_p}$ is called the asymptotic shape. At this stage, it seems natural to wonder if the asymptotic shape converges with respect to the law  $p$ for the Hausdorff distance when  $\mu_p(\e_1)>0$. In fact, this result follows from Theorem~\ref{uni}. Let us begin by recalling the definition of the Hausdorff distance. Let $X$ and $Y$ be two non-empty compact sets of a metric space $(E,d)$. We define their Hausdorff distance $d_H(X, Y)$ by
$$d_H \left(X,Y\right):= \max \left\lbrace  \sup_{y\in Y} \inf_{x \in X} d(x,y), \sup_{x\in X} \inf_{y \in Y} d(x,y) \right\rbrace.$$

\begin{thm}\label{hauss}
 For each $p$ in $\mathcal{P}$ such that the asymptotic shape is compact, we have
$$\lim_{q \rightarrow p} d_H\left(B_{\mu_p},B_{\mu_q}\right) =0.$$
\end{thm}

\smallskip

The article is organized as follows.

\begin{itemize}
\item In Section~\ref{debut}, we recall the condition of strict positivity of $\mu$ and some results on greedy lattice animals. Then we define a notion of short paths, which leads to consider a sequence $(\mu^k)_{k \in \N^*}$ of intermediate functions $\mu^k$ and we show the continuity of the functions $\mu^k$, for all positive integer $k$.
\item Section~\ref{demo} is devoted to the proofs of the theorems. We show that $(\mu^k)_{k \in \N^*}$ converges uniformly to $\mu$ on each set of an open covering of the coloring law and we use the continuity of $\mu$ to get the continuity of the asymptotic shape with respect to the coloring law for the Hausdorff distance.
\item In the last section is postponed the proof of the upper bound of the probability to have paths rather large but of reasonable passage time. It is a useful lemma for the proof of the convergence of $(\mu^k)_{k \in \N^*}$ to $\mu$ and the proof is based on greedy lattice animals.
\end{itemize}

\section{Preliminary results}
\label{debut}

\subsection{Positivity condition}

In their article~\cite{fetn}, Fontes and Newman have given a necessary and sufficient condition for the strict positivity of $\mu(\e_1)$.

\begin{prop}[Fontes and Newman, 93] \label{positivity} In the case of first passage percolation on a random coloring, we have
$$\mu_p(\e_1) >0 \text{ if and only if } |p| < p_c(\Z^d, \text{site}),$$
where $p_c(\Z^d, \text{site})$ is the critical probability for the Bernoulli site percolation model on $\Z^d$ (see Grimmett~\cite{grimmett-book}).
\end{prop}

\begin{req} In fact, we can see that
\begin{enumerate}
 \item if $\mu(\e_1)>0$, then for each integer $i$ between $1$ and $d$, $\mu_p(\e_i) >0$ and $\mu_p$ is a norm on $\R^d$,
\item if $\mu(\e_1)=0$, then for each integer $i$ between $1$ and $d$, $\mu_p(\e_i) =0$ and, for all $x$ in $\R^d$, $ \mu_p(x)\leq\|x\|  \mu_p(\e_1)=0$.
\end{enumerate}
Thus for each $x$ in $\mathbb{R}^d \setminus \{\textbf{0}\}$, $\mu_p(x)>0$ if and only if $|p| < p_c(\Z^d, \text{site})$.
\end{req}

\medskip

Their proof is based on greedy lattice animals. This model was introduced by Cox, Gandolphi, Griffin and Kesten~\cite{gla1}.

Let $ \left\{X_v \right\}_{v \in \Z^d}$ be an i.i.d. family of non-negative random variables. For a finite subset $\xi$ of $\Z^d$, the \textit{weight} $S \left( \xi \right)$ of $\xi$ is defined by $S \left( \xi \right) = \sum_{v \in \xi} X_v$. 
A \textit{greedy lattice animal of size $n$} is a connected subset of $\Z^d$ of size $n$ containing the origin whose weight is maximal among all such sets. This maximum weight is denoted by $W(n)~:=~\sup_{|\xi|=n} \sum_{v \in \xi} X_v$. 

To study percolation, we are interested in the limit of $\frac{W(n)}{n}$ which is related to the renormalized supremum of passage time all over the paths containing exactly $n$ sites. Gandolfi and Kesten~\cite{gla2} show that, under the condition that, $\textbf{E}X_0^d (\log^+ X_0)^{d+\epsilon} < \infty$ for some $\e$, there exists a constant $W < \infty$ such that
$$\frac{W(n)}{n} \rightarrow W \text{ almost surely and in } \mathcal{L}_1.$$

In his paper, James Martin~\cite{Martin} shows the same results under slighty weaker condition and, during the proof of Lemma 6.5 (see point (6.14)), he obtains a large deviation result when the law of the $X_v$ is bounded.

\begin{prop}[Martin, 02]\label{martthm}
Let $y$ be a positive real number.

 Then there exist positive real numbers $C_1$, $C_2(y,d)$, $N(y,d)$ such that for each integer $n \geq N$ and for each i.i.d family $\left\{X_v \right\}_{v \in \Z^d}$ of non-negative random variables bounded by $y$, we have
\begin{equation*} \label{mart}
 \textbf{P} \left( \frac{W(n)}{n} \geq W+ 1 \right) \leq C_1 e^{ - C_2 \frac{n}{y^2}}.
\end{equation*}
\end{prop}

These results will be used in Section~\ref{kesten} to prove exponential decay of the probability to have paths rather large but of reasonable passage time.

\subsection{Seminorms on $k$-short paths}

To study the seminorm $\mu$, we introduce the notion of $k$-short paths and of seminorms on $k$-short paths. It will be more convenient to work with them and the new seminorms get closer to $\mu$.

\subsubsection{Definitions and relation to $\mu$}

For each positive integer $k$, a \textit{$k$-short path} between two vertices $u$ and $v$ is a path $\gamma$ between $u$ and $v$ such that $l(\gamma) \leq k \|u-v\|$. The set of the $k$-short paths between $u$ and $v$ is denoted by $\Gamma_k(u,v)$ t. These definitions lead to the notion of passage time on $k$-short paths: for each positive integer $k$ and for each $x,y$ in $\R^ d$, we define $$T_p^k(x,y):=\inf\left\lbrace T_p(\gamma); \gamma \in \Gamma_k( x,y) \right\rbrace.$$
\begin{lem}\label{muk}
For each $p$ in $\mathcal{P}$, there exist functions $\mu_p^1, \mu_p^2, \ldots$ from $\mathbb{R}^d$ to $\R_+$ such that, for each positive integer $k$ and for each $x$ in $\mathbb{R}^d$, we have 
\begin{equation} \label{limk1}
\lim_{n \rightarrow + \infty} \frac{T_p^k(0,nx) }{n}= \mu_p^k(x)\text{ a.s.}
\end{equation}
Moreover, for each $x$ in $\R^d$, we have 
\begin{equation} \label{limk2}
\lim_{k \to + \infty} \mu_p^k(x)= \inf_{k \in \N^*} \mu_p^k(x)=\mu_p(x).
\end{equation} 
\end{lem}

\begin{proof}[Proof of Lemma~\ref{muk}]
We divide the proof into three parts. We begin by proving the existence and some properties of the functions $\mu^ k_p$ on $\Z^d$. Then we extend the functions $\mu^ k_p$ by continuity on $\R^d$ and we show that \eqref{limk1} is still satisfied, first on $\mathbb{Q}^d$, then on $\R^d$. We finish with the demonstration of the second point~\eqref{limk2}. Again we work on $\Z^d$ and we extend the result on $\R^d$, \textit{via} $\mathbb{Q}^d$.

Let $x$ in $\mathbb{Z}^d$. For each integer $k$, we can apply the subadditive ergodic theorem to the family of random variables $\big(T_p^k(mx,nx) \big)_{0 \leq m <n}$. Thus, for each $x$ in $\mathbb{Z}^d$, we have the existence of $\mu_p^k(x)$ as the limit of $\frac{T_p^k(0,nx) }{n}$. In fact, the limit is the lower bound and is also the lower bound of the expectation. We can extend the function $\mu_p^k$ to all $\R^d$, as we did with $\mu_p$. Like $\mu_p$, the function $\mu_p^k$ is a 1-Lipschitz continuous seminorm.

Now we check that, for all $x$ in $\R^d$, $\lim_{n \rightarrow + \infty} \frac{ T_p^k(0,nx)}{n}=\mu_p^k(x)$ a.s.

We fix a positive integer $k$. For all $x$ in $\R^d$, we put $f_n(x)= \frac{T_p^k(0,nx)}{n}$. We proceed in two steps to study the convergence of the sequence $\left( f_n(x) \right)_{n \in \N^*}$, for all $x$ in $\R^d$. We first show the result on $\mathbb{Q}^d$, then we extend the result to $\R^d$. 

Let $x$ be in $\mathbb{Q}^d$. There exist $\hat{x}$ in $\Z^d$ and a positive integer $\lambda$ such that $x=\frac{\hat{x}}{\lambda}$. As $\hat{x}$ is in $\Z^d$, we already have $\lim_{n \rightarrow + \infty} f_n(\hat{x}) = \mu_p^k(\hat{x})$. 

Now we are going to prove that $\lim_{n \rightarrow + \infty} f_n( x)=\mu_p^k( x)=\frac{1}{\lambda} \mu_p^k(\hat{x})$.

\begin{align*}
 \bigg| \frac{T_p^k(0,n x)}{n}- \frac{T_p^k(0, \left\lfloor \frac{n}{\lambda} \right\rfloor \hat{x})}{n} \bigg| & \leq \frac{T_p^k(n x,\left\lfloor \frac{n}{\lambda} \right\rfloor \hat{x} )}{n} \leq \frac{1}{n} \left( \left\| n x - \left\lfloor \frac{n}{\lambda} \right\rfloor \hat{x} \right\| +d \right) \\
& \leq \frac{1}{n} \left( \bigg| \frac{n}{\lambda} - \left\lfloor \frac{n}{\lambda} \right\rfloor \bigg| \| \hat{x} \| +d \right) \leq \frac{\| \hat{x} \| +d }{n} \rightarrow 0
\end{align*}

Moreover, $
 \frac{T_p^k(0, \left\lfloor \frac{n}{\lambda} \right\rfloor \hat{x})}{\left\lfloor \frac{n}{\lambda} \right\rfloor} \times \left\lfloor \frac{n}{\lambda} \right\rfloor \times \frac{\lambda}{n} \times \frac{1}{\lambda} \rightarrow \mu_p^k(\hat{x}) \times \frac{1}{\lambda} = \mu_p^k(x).
$

Thus, for all $x$ in $\mathbb{Q}^d$, $\mu_p^k(x) = \lim_{n \rightarrow + \infty} \frac{T_p^k(0,nx)}{n}$.

Now we extend the result to $\R^d$. Let us first remark that, for all $x$ and $y$ in $\R^d$, we have
$$| f_n(x) -f_n(y) | \leq \frac{T_p^k(nx,xy)}{n} \leq \frac{\| nx-ny \| +2 }{n} \leq \| x -y \| + \frac{2}{n}.$$ 
Let $x$ be in $\R^d$ and $\e >0$. There exists $x_{\e}$ in $\mathbb{Q}^d$ such that $\|x- x_{\e} \| < \e$. Thus we have
\begin{align*}
|f_n(x) - \mu_p^k(x) | & \leq |f_n(x) - f_n(x_{\e})| + |f_n(x_{\e}) - \mu_p^k (x_{\e}) | + |\mu_p^k(x_{\e} - \mu_p^k(x)| \\
& \leq 2 \e + \frac{2}{n} + | f_n(x_{\e}) - \mu(x_{\e})|.
\end{align*}
Thanks to the result on $\mathbb{Q}^d$, there exists a positive integer $N$ such that, for each integer $n \geq N$, we have $|f_n(x) - \mu_p^k(x) | \leq 3 \e$.

Thus, for all $x$ in $\mathbb{R}^d$, we have $\mu_p^k(x) = \lim_{n \rightarrow + \infty} \frac{T_p^k(0,nx)}{n}$.

\medskip
It remains to prove the convergence of the sequence $\left( \mu_p^k(x) \right)_{k \in \N^*}$ to $\mu_p(x)$, for each $x$ in $\R^d$.

Fix an integer $k$. We have $T_p^k(0,nx)~\geq~T_p^{k+1}(0,nx),$ thus  $\textbf{E}T_p^k(0,nx)~\geq~\textbf{E}T_p^{k+1}(0,nx)$ and $\mu^k_p(x) \geq \mu_p^{k+1}(x)$.
Since the sequences $\left( \mu_p^k(x) \right)_{k \in \N^*}$ and $\left( \textbf{E}T_p^k(0,nx) \right)_{k \in \N^*}$ are monotonically decreasing and bounded by below, they converge and we have 
$$\lim_{k \rightarrow + \infty} \mu_p^k(x)=\inf_{k \in \N^*} \mu_p^k(x) \text{ and } \lim_{k \rightarrow + \infty} \textbf{E}\left[ T_p^k(0,nx) \right]= \inf_{k \in \N^*} \textbf{E}\left[ T_p^k(0,nx) \right].$$

Furthermore, we have $\lim_{k \rightarrow + \infty} \textbf{E}\left[ T_p^k(0,nx) \right] = \textbf{E}[T_p(0,nx)]=\inf_{k \in \N^*} \textbf{E}\left[ T_p^k(0,nx) \right]$, by the dominated convergence theorem.

Therefore, by Kingman's subadditive ergodic theorem, for each $x$ in $\Z^d$, we have
\begin{align*}
 \lim_{k \to + \infty} \mu_p^k(x)&  = \inf_{k \in \N^*} \mu_p^k(x) = \inf_{k \in \N^*} \inf_{n \in \N^*} \frac{\textbf{E}\left[T_p^k \left(0,n x \right) \right]}{n }  = \inf_{n \in \N^*} \frac{1}{n} \inf_{k \in \N^*} \textbf{E}\left[T_p^k \left(0,n x \right) \right]\\
& = \inf_{n \in \N^*} \frac{\textbf{E}\left[T_p \left(0,n x \right) \right]}{n}=\mu_p(x).
\end{align*}

Let $x$ be in $\R^d$ and $\e >0$. There exists $\hat{x}_{\e}$ in $\mathbb{Q}^d$ such that $\|x - \hat{x}_{\e} \| \leq \e$. Then there exists a positive integer $\lambda$ and $x_{\e}$ in $\Z^d$ such that $\hat{x}_{\e}=\frac{x_{\e}}{\lambda}$. Since $\mu_p$ and $\mu_p^k$ are seminorms, we have
\begin{align*}
 | \mu_p^k(x) - \mu_p(x) | & \leq | \mu_p^k(x) - \mu_p^k(\hat{x}_{\e}) | +| \mu_p^k(\hat{x}_{\e}) - \mu_p(\hat{x}_{\e}) | +| \mu_p(\hat{x}_{\e}) - \mu_p(x) | \\
& \leq 2 \e + \frac{1}{\lambda} | \mu_p^k(x_{\e}) - \mu_p(x_{\e}) |.
\end{align*}

Thanks to the previous result, there exists a positive integer $K$ such that for all integer $k>K$, we have $| \mu_p^k(y) - \mu_p(y) | \leq 3 \e$. 
Therefore we have $\lim_{k \to + \infty} \mu_p^k(x)=\mu_p(x)$, for each $x$ in $\R^d$.

\end{proof}

\subsubsection{Continuity of the function $\mu^k$}

\begin{lem}\label{contmuk}
For all positive integer $k$, the function $\mu^k$ :
\begin{tabular}{ccc}
 $\mathcal{P}$ & $\rightarrow$ & $\mathcal{C} \left( \mathcal{B}(0,1) , \R \right) $ \\
$p$ & $\mapsto$ & $\mu^k_p$
\end{tabular} is continuous with respect to the infinity norm.
\end{lem}

It follows that the function \begin{tabular}{ccc}
 $\mathcal{P}$ & $\rightarrow$ & $\R_+$ \\
$p$ & $\mapsto$ & $\mu^k_p(x)$
\end{tabular} 
is continuous, for all positive integer $k$ and for each $x$ in $\R^d$.

\begin{proof}[Proof of Lemma~\ref{contmuk}]

The idea of the proof is to proceed by coupling. Let $\big(U_u \big)_{u \in \Z^d}$ be a family of independent uniform random variables on $[0;1]$ and $p=(p_1, \ldots)$ in $\mathcal{P}$. For all $u$ in $\mathbb{Z}^d$, we define $$X_u^p:= \sum_{i=1}^{+ \infty} i \textbf{1}_{\{p_0+ \cdots + p_{i-1} \leq U_u <p_1 + \cdots +p_i\}}$$ with $p_0=0$ by convention. By this way, the random family $(X_u^p)_{u \in \Z^d}$ corresponds to the model of first passage percolation attached to a random coloring with law $p$. By using this coupling, we will prove that, if $p$ and $q$ are two very close elements of $\mathcal{P}$, then the probability $\textbf{P} \big(X_u^p \neq X_u^q \big)$ is very small. So $\mu^k_p(x)$ and $\mu^k_q(x)$ are very close too.

\medskip

Let $p$ in $\mathcal{P}$, $x$ in $\mathcal{B}(0,1)$, $\e>0$ and $k$ be a positive integer.

Put $\delta := \delta(\e,x,k)= (4d)^{- \frac{3 k \|x \|}{\e}}$. We choose a positive integer $S$ such that $\sum_{i=S}^{+ \infty} p_i < \frac{\delta}{2}$.

Let $q$ in $\mathcal{P}$ such that $|p-q| < \frac{\delta}{(S-1)S}$. By the previous coupling, we have
\begin{align*}
\textbf{P} \big(X^p(u) \neq X^{q}(u) \big) & = \sum_{i=1}^{+ \infty} \textbf{P} \left(X^p(u)=i \text{ and } X^q(u) \neq i \right)\\
 & = \sum_{i=1}^{+ \infty} \textbf{P} \left(U(u) \in \left[ \sum_{j=0}^{i-1}p_{j} , \sum_{j=0}^{i}p_j \right[ \text{ and } U(u)\notin \left[\sum_{j=0}^{i-1}q_{j} , \sum_{j=0}^{i}q_j \right[  \right) \\
& \leq \sum_{i=1}^{S} \left( \bigg|\sum_{j=0}^{i-1}p_{j} -\sum_{j=0}^{i-1}q_j\bigg| + \bigg| \sum_{j=0}^{i}p_j - \sum_{j=0}^{i}q_{j} \bigg| \right) + \textbf{P} \left( U_u \geq \sum_{j=0}^{S}p_j \right)  \\
& \leq  \sum_{i=1}^{S} \left( 2|p_0-q_0| + \cdots +2|p_{i-1}-q_{i-1}| +|p_i-q_i| \right) + \sum_{i=S+1}^{+ \infty} p_i \\
& \leq 2 \sum_{i=1}^{S} (S+1-i) |p_i-q_i| + \sum_{i=S+1}^{+ \infty} p_i \\
& \leq 2 |p-q| \sum_{i=1}^{S} i + \frac{\delta}{2} \leq 2 \frac{\delta}{(S+1)S} \frac{S(S+1)}{2} + \frac{\delta}{2} =  \delta.
\end{align*}

\smallskip

Thereby we have

\begin{align*}
\textbf{P} \big( \big|T^k_{q} & (0,nx) - T^k_p(0,nx) \big|  \geq \e n \big) \\ 
& \leq \textbf{P}\left( \exists \gamma \in \Gamma(0, \cdot); l(\gamma) \leq kn , \exists \text{ at least } \left\lfloor \frac{\e n}{2} \right\rfloor \text{ vertices $u$ } \in \gamma , X_q(u) \neq X_p(u) \right) \\
& \leq \sum_{\gamma \in \Gamma(0, \cdot), \ |\gamma| \leq kn} \textbf{P} \left( \exists \text{ at least } \left\lfloor \frac{\e n}{2} \right\rfloor \text{ vertices $u$ } \in \gamma,X_q(u) \neq X_p(u)\right) \\
& \leq (2d)^{kn} \dbinom{kn}{\left\lfloor \frac{\e n}{2} \right\rfloor}\textbf{P} \left(X^p(u) \neq X^{q}(u) \right)^{\frac{\e n}{2}-1} \leq (2d)^{kn} 2^{kn}\textbf{P} \left(X^p(u) \neq X^{q}(u) \right)^{\frac{\e n}{2}-1} \\
& \leq (4d)^{nk} (4d)^{- \frac{3k}{\e} \left( \frac{\e n}{2}-1\right) } = (4d)^{\frac{3k}{\e}} \left( (4d)^{k- \frac{3k}{\e}  \frac{\e}{2}} \right)^{n}= (4d)^{\frac{3k}{\e}} \left( (4d)^{-\frac{k}{2}} \right)^{n}.
\end{align*}

Thus, for all positive integers $n$, $\textbf{P}\Big( \big|T^k_{q}(0,nx) - T^k_p(0,nx) \big|  \geq \e n \Big) \leq (4d)^{\frac{3k}{\e}} \left( (4d)^{-\frac{k}{2}} \right)^{n}.$

Therefore, $\sum_n \textbf{P}\Big( \big|T^k_{q}(0,nx) - T^k_p(0,nx) \big|  \geq \e n \Big) <  \infty.$

By Borel--Cantelli's lemma, we have $\big|T^k_{q}(0,nx) - T^k_p(0,nx) \big|  < \e n$ for $n$ large enough a.s.

\smallskip

Thus for all positive integers $k$, we have  $|\mu^k_q(x) - \mu^k_p(x)| < \e$. 

\end{proof}

\section{Proofs of Theorem~\ref{uni} and Theorem~\ref{hauss}}

\label{demo}
\subsection{Seminorm continuity}

We begin by showing that, for each $x$ in $\R^d$, the function $p \mapsto \mu_p(x)$ is continuous when $|p|\geq p_c$. In fact, only the continuity of the function $p \mapsto \mu_p(\e_1)$ will be useful. Then we prove separately the continuity of the function $p \mapsto \mu_p$ on the set $\mathcal{P}_{p_c}:=\left\lbrace p \in \mathcal{P} \ ; \ |p|<p_c \right\rbrace$ and on the set $\mathcal{P} \setminus \mathcal{P}_{p_c}$.\\
\paragraph{\textbf{Pointwise continuity on $\mathcal{P} \setminus \mathcal{P}_{p_c}$}}
Since, for each $x$ in $\R^d$, $\mu_p(x)$ is the infinimum of $\left(\mu_p^k(x)\right)_{k \in \N}$ and the function $\mu_p^k(x)$ is continuous with respect to $p$, then the function $\mu_p(x)$ is upper semi-continuous, $ \text{\textit{i.e.}  } \limsup_{q \rightarrow p} \mu_q(x) \leq \mu_p(x)$. Thus if $p$  is such that $|p| \geq p_c(\text{site})$, then we have
$0 \leq \liminf_{q \rightarrow p} \mu_q(x) \leq \limsup_{q \rightarrow p} \mu_q(x) \leq \mu_p(x)=0.$
Hence $\lim_{q \rightarrow p} \mu_q(x) =0= \mu_p(x).$ Thus the function $p \mapsto \mu_p(x)$ is always continuous on the set $\mathcal{P} \setminus \mathcal{P}_{p_c}$.\\
\paragraph{\textbf{Uniform continuity on the set $\mathcal{P} \setminus \mathcal{P}_{p_c}$}}
For each $x$ in $\mathcal{B}(0,1)$ and each $q$ in $\mathcal{P}$, we have $|\mu_q(x)| \leq \sum_{i=1}^d |x_i| |\mu_q(\e_i)|$. So when $p$ is such that $|p| \geq p_c$, we have $$\lim_{q \rightarrow p} \sup_{x\in B(0,1)} | \mu_p(x) -\mu_q(x)| \leq \lim_{q \rightarrow p} \sup_{x\in B(0,1)}  \sum_{i=1}^d |x_i| |\mu_q(\e_i)| \leq d \lim_{q \rightarrow p} |\mu_q(\e_1)|=\mu_p(\e_1)=0.$$

Thus $\lim_{q \rightarrow p} \sup_{x\in B(0,1)} | \mu_p(x) -\mu_q(x)|=0$, hence the continuity of the function $p \mapsto \mu_p$ on $\mathcal{P} \setminus \mathcal{P}_{p_c}$.\\

\paragraph{\textbf{Uniform continuity on the set $\mathcal{P}_{p_c}$}}
To prove the continuity of the function $\mu$ on the set $\mathcal{P}_{p_c}$, we introduce a family of sets $E_{\theta, S}$ defined, for each real number $\theta$ in $(0, p_c(\text{site}, \Z ^d))$ and each positive integer $S$, by
 $$E_{\theta,S}:=\left\lbrace p \in \mathcal{P} \ ; \ |p|<\theta, \ \sum_{i=S+1}^{+ \infty} p_i <\theta \right\rbrace. $$
 
 \begin{lem}\label{rec}
 The family $\big(E_{\theta,S} \big)_{0 <\theta <p_c, \ S \geq 5}$ is an open covering of the set $\mathcal{P}_{p_c}$.
 \end{lem}

\begin{proof}[Proof of Lemma~\ref{rec}]
Let $0 <\theta <p_c$, $S \geq 5$ and $p$ in $E_{\theta,S}$. 
Let $\phi$ be the application defined on $\mathcal{P}$ by $\phi(p) = \sup_i|p_i|$ and $\psi_S$ be the application defined on $\mathcal{P}$ by $\psi_S(p)= \sum_{i=S+1}^{+ \infty} p_i$. So $E_{\theta,S} = \phi^{-1} \left( [0, \theta [ \right) \cap \psi_S^{-1} \left( [0, \theta [ \right)$.
However, for all $p$ and $q$ in $\mathcal{P}$, we have
$$|\phi(p)-\phi(q)| \leq \big| |p| - |q| \big| \leq \big| p-q \big| \leq \|p-q\|_1$$
and
$$| \psi_S(p) - \psi_S(q) | \leq \Big| \sum_{i=S+1}^{+ \infty} |p_i| - \sum_{i=S+1}^{+ \infty} |q_i| \Big| \leq \sum_{i=S+1}^{+ \infty} |p_i-q_i| \leq  \|p-q\|_1.$$
Thus the functions $\phi$ and $\psi_S$ are continuous. Thus, for all $0 <\theta <p_c$ and $S \geq 5$, $E_{\theta,S}$ is an open set.

\smallskip

It remains to prove that $\bigcup_{0<\theta < p_c}~\bigcup_{S \geq 5}~E_{\theta,S} = \mathcal{P}_{p_c}$.
We clearly have $\bigcup_{0<\theta < p_c}~\bigcup_{S \geq 5}~E_{\theta,S}~\subset~\mathcal{P}_{p_c}$.
Then for all $p$ in $\mathcal{P}_{p_c}$, there exist a real number $\theta$ in $(0, p_c)$ and a positive integer $S$ larger than 5 such that $|p|< \theta$ and $\sum_{i=S+1}^{+ \infty} p_i < \theta$. Thus $p$ is in $E_{\theta,S} \subset \bigcup_{0<\theta < p_c} \bigcup_{S \geq 5} E_{\theta,S}$.

\medskip

Therefore the family $\big(E_{\theta,S} \big)_{0 <\theta <p_c, \ S \geq 5}$ is an open covering of the set $\mathcal{P}_{p_c}$.
\end{proof}

Now we prove that on each set $E_{\theta, S}$, the sequence $\left(\mu_p^k \right)_{k \in \N^*}$ converges uniformly to the function $\mu_p$. By this way, we will have proved \ref{uni}.

\begin{lem} \label{convuni}
Let $\theta$ be a real number in $(0, p_c)$ and $S$ be an positive integer larger than $5$. There exist positive real numbers $K(\theta, S), \ C_1(\theta, S)$ and $C_2(\theta, S)$, such that for each integer $k> K$,
$$\|\mu_p^ k - \mu_p\|_{\infty} \leq C_1 e^ {-C_2k^{\frac{1}{5}}}, \ \forall p \in E_{\theta, S}.$$
\end{lem}

\begin{proof}[Proof of Lemma~\ref{convuni}]
Let $\theta$ be in $(0, p_c)$, $S$ be an integer larger than $5$ and $p$ in $E_{\theta, S}$. We consider $x$ in $\mathcal{B}(0,1)$.

We know that
$\mu^k_p(x)-\mu_p(x) = \lim_{n \rightarrow + \infty} \textbf{E}\left[ \frac{T_p^k(0,nx)}{n} \right] - \lim_{n \rightarrow + \infty} \textbf{E}\left[\frac{T_p(0,nx)}{n} \right] \geq 0.$

Thus $0 \leq \mu^k_p(x)-\mu_p(x) \leq \sup_{n \geq 1} \textbf{E}\left[ \frac{T_p^k(0,nx)}{n}- \frac{T_p(0,nx)}{n} \right]$.
Since for all positive integers $n,k$, $T_p^k(0,nx) \geq T_p(0,nx)$, then we have
\begin{align*}
\textbf{E}[T_p^k(0,nx)] & \leq \textbf{E}[T_p(0,nx)] + \textbf{E}[T_p^k(0,nx) \textbf{1}_{\{T_p(0,nx) < T_p^k(0,nx)\}}]  \\
&\leq  \textbf{E}[T_p(0,nx)] + \left( n ||x|| +d \right) \textbf{P} \big( T_p(0,nx) < T_p^k(0,nx) \big)
\end{align*}

because the passage times are bounded from above by $1$ and there exists a deterministic path of length lower than $n||x||+d$ from $0$ to $n x$. Thus we obtain $$\mu^k_p(x)-\mu_p(x) \leq \sup_{n \geq 1} \textbf{P} \big( T_p(0,nx) < T_p^k(0,nx) \big) \times \|x\| \leq \sup_{n \geq 1} \textbf{P} \big( T_p(0,nx) < T_p^k(0,nx) \big).$$

Note that
\begin{align*}
 \textbf{P} \big( T_p(0,nx) < T_p^k(0,nx)\big) & \leq \textbf{P}  \big( \exists \gamma \in \Gamma(0,nx), \ l(\gamma) > nk, \ T_p(\gamma)=T_p(0,nx) \big)\\
& \leq \textbf{P}  \big( \exists \gamma \in \Gamma(0,nx), \ |\gamma| > nk, \ T_p(\gamma)=T_p(0,nx) \big)\\
& \leq \textbf{P}  \big( \exists \gamma \in \Gamma(0,nx), \ |\gamma| > nk, \ T_p(\gamma) \leq n \big) \\
& \leq \textbf{P}  \big( \exists \gamma \in \Gamma(0,\cdot), \ |\gamma| \geq nk, \ T_p(\gamma) \leq n \big)\\
& \leq \textbf{P}  \big( \exists \gamma \in \Gamma(0,\cdot), \ |\gamma| = nk, \ T_p(\gamma) \leq n \big). 
\end{align*}

Yet we have the following result whose proof is put back to Section~\ref{kesten}.
\begin{lem} \label{petitc}
Let $\theta$ be a real number in $(0, p_c)$ and $S$ be an integer larger than $5$.\\
There exist positive real numbers $N(\theta, S), K(\theta, S),C_1(\theta, S),C_2(\theta, S)$ such that, for each positive integers $n \geq N$ and $k \geq K$, we have
\begin{equation*}
\label{expok}
\textbf{P}_p \big( \exists \gamma \in \Gamma (\textbf{0}, \cdot) \ ; \ |\gamma| = nk, \ T(\gamma) \leq n \big) \leq C_1 e^{-C_2 (nk)^{\frac{1}{5}}}, \ \forall p \in E_{\theta,S}.
\end{equation*}
\end{lem}

By this lemma, there exist  positive real numbers $K(\theta,S), \ C_1(\theta,S), C_2(\theta,S)$ such that, for each positive integer $k > K,$ we have
\begin{equation*} 
\sup_{x \in B(0,1)} |\mu^k_p(x)-\mu_p(x)| \leq \sup_{n \geq 1} C_1 e^{-C_2 (nk)^{\frac{1}{5}}} \leq C_1 e^{-C_2 k^{\frac{1}{5}}}, \text{ for all $p$ in $E_{\theta,S}$}.
\end{equation*}
\end{proof}

Therefore we have the uniform convergence of the sequence $\left( \mu^k\right) _{k \in \N^*}$ to the function $\mu$, which completes the proof of Theorem~\ref{uni}.

\subsection{Convergence for Hausdorff distance}
\label{haus}

By Proposition~\ref{positivity}, the asymptotic shape is compact when $p$ is in the set $\mathcal{P}_{p_c}$, \textit{i.e.} $|p|<p_c$. Thus we will show the continuity of the asymptotic shape with respect to $p$ for the Hausdorff distance on $\mathcal{P}_{p_c}$.

\begin{proof}[Proof of Theorem~\ref{hauss}]
The function $\mu_p$ is continuous. Thus, on the compact $\{ \|x\| =1 \}$, the function $\mu_p$ attains its minimum value which is positive because $\mu_p(x)$ equals zero if and only if $x=0$, when $|p|<p_c$. 

We fix $p$ in $\mathcal{P}_{p_c}$ and $\e$ in $(0,1)$.
As the function $\mu$ is continuous, there exists a positive real number $\eta_1:= \eta_1(p,d)$ such that, if $q$ in $\mathcal{P}_{p_c}$ satisfies $|p-q|< \eta_1$, then $| \alpha_p - \alpha_q|< \frac{\alpha_p}{2}$. Therefore $\frac{1}{\alpha_q}  \leq \frac{2}{\alpha_p}$.
By Theorem~\ref{uni}, we have $\lim_{q \rightarrow p} \sup_{x \in \mathcal{B}(0,1)} \big| \mu_p(x) - \mu_{q}(x) \big|=0.$ Thus there exists a positive real number $\eta_2~:=~\eta_2~(\e,p,d)$ such that, if $q$ vsatisfies $|p- q| <\eta_2 $, then $\sup_{x \in \mathcal{B}(0,1)} \big| \mu_p(x) - \mu_{q}(x) \big|<  \frac{\e \alpha_{p}^2}{2}$.

We choose $\eta$ lower than $\eta_1$ and $\eta_2$, then we fix $q$ in $\mathcal{P}_{p_c}$ such that  $|p-q|< \eta$.
\smallskip

Therefore we have
\begin{align*}
 d_H \left( B_{\mu_q}, B_{\mu_p} \right) & \leq \sup_{\|x\|=1} d \left( \frac{x}{\mu_p(x)} , \frac{x}{\mu_q(x)} \right) \leq \sup_{\|x\|=1} \left\lbrace  \|x\| \frac{ \big| \mu_q(x) - \mu_p(x) \big|}{\mu_p(x) \mu_q(x)} \right\rbrace \\
& \leq \frac{1}{\alpha_p \alpha_q} \sup_{y \in \mathcal{B}(0,1)} \left\lbrace \big|  \mu_q(y) - \mu_p(y) \big| \right\rbrace \leq \frac{\e \alpha_{p}^2}{2} \times \frac{2 }{\alpha_p ^2} = \e.
\end{align*}
\end{proof}

\section{Upper bound for $\textbf{P}  \big( \exists \gamma \in \Gamma(0,n), \ |\gamma| = nk,  \ T^p(\gamma) \leq n \big)$}
\label{kesten}

Let us now proove Lemma~\ref{petitc}. Let $S$ be a positive integer larger than $5$, $\theta$ be a real number in $(0, p_c)$ and $p$ be in the set $E_{\theta,S}$. 
Fix $m$ and $k$ two positive integers.

First, note that
\begin{align*}
 \textbf{P} \big( \exists r \in \Gamma (0,\cdot),  \ |r| = mk, \ T(r) \leq m \big) & = \textbf{P} \left( \exists r \in \Gamma (0,\cdot),  \ |r| =n, \ T(r) \leq \frac{n}{k} \right) \text{ with } n=mk\\
& = \textbf{P} \left( \inf_{r \in \Gamma (0,\cdot),  \ |r|=n}  T(r) \leq \frac{n}{k} \right)\\
& = \textbf{P}\left(1 + \inf_{r \in \Gamma (0,\cdot),  \ |r|=n}  T(r) \leq \frac{n}{k} +1 \right)\\
& = \textbf{P}\left( \frac{1 + \inf_{r \in \Gamma (0,\cdot),  \ |r|=n}  T(r)}{n} \leq \frac{1}{k} + \frac{1}{n} \right).
\end{align*}

To study $\frac{1 + \inf_{r \in \Gamma (0,\cdot),  \ |r|=n}  T(r)}{n}$, we do the same remark as Fontes and Newman did in their article~\cite{fetn}. For any (self-avoiding) path starting from the origin and containing exactly $n$ sites, we have
\begin{align*}
1 + T(r) &   \geq \text{ number of distinct color clusters touched by } r  = \sum_{v \in r} | \C_v \cap r |^{-1},\\
\intertext{where $\C_v$ denotes the color cluster of vertex $v$ and $|A|$ denotes the number of vertices in $A$. Then, by using Jensen's inequality, we obtain}
\frac{1 + T(r)}{n} & = \frac{1 + T(r)}{|r|} \geq \frac{1}{|r|} \sum_{v \in r} | \C_v \cap r |^{-1}  \geq \left[  \frac{1}{|r|} \sum_{v \in r} | \C_v \cap r | \right]^{-1} \geq \left[ \frac{1}{|r|} \sum_{v \in r} | \C_v | \right]^{-1}. \end{align*}
For each color $s$ in $\N^*$, we define the color-$s$ cluster at vertex $v$ by 
\begin{equation*}\text{for } s \leq S, \ \C_v^s= \begin{cases}
                           \C_v, & \text{if } X_v=s,\\
                            \emptyset, &  \text{if } X_v \neq s,
                          \end{cases}
\text{ and } \C_v^{S+1}= \begin{cases}
                           \C_v, & \text{if } X_v \geq S+1 ,\\
                            \emptyset, &  \text{if } X_v < S+1. 
                          \end{cases}\end{equation*}
We identify all the colors larger than $S+1$.
With this notation, we have
$$\frac{1 + T(r)}{n} \geq \left[ \frac{1}{|r|} \sum_{v \in r} \sum_{s=1}^ {S+1} | \C_v^s | \right]^{-1}. $$
Consequently,
$$\frac{1 + \inf_{|r|=n}T(r)}{n} \geq \inf_{|r|=n} \left[ \frac{1}{|r|} \sum_{v \in r} \sum_{s=1}^ {S+1} | \C_v^s | \right]^{-1} = \left[ \sup_{|r|=n} \frac{1}{|r|} \sum_{v \in r} \sum_{s=1}^ {S+1} | \C_v^s | \right]^{-1},$$
where the sup is over all self-avoiding paths $r$ starting from the origin an containing exactly $n$ sites.

Thus, with $\alpha:=\frac{1}{k} + \frac{1}{n}$, we have
\begin{align*}
\textbf{P} \big( \exists r \in \Gamma (0,\cdot),  \ |r| = mk, \ T(r) \leq m \big) & \leq \textbf{P} \left( \sup_{|r|=n} \frac{1}{|r|} \sum_{v \in r} \sum_{s=1}^ {S+1} | \C_v^s | \geq \frac{1}{\alpha} \right)\\
& \leq \textbf{P} \left( \sum_{s=1}^ {S+1} \sup_{|r|=n} \frac{1}{|r|} \sum_{v \in r}  | \C_v^s | \geq \frac{1}{\alpha} \right)\\
& \leq \sum_{s=1}^ {S+1} \textbf{P} \left(  \sup_{|r|=n} \frac{1}{|r|} \sum_{v \in r}  | \C_v^s | \geq \frac{1}{\alpha(S+1)} \right).
\end{align*}

Yet we have the following stochastic domination inequality $\left(|\C_v^s|\right)_{v \in \Z^d} \prec \left(| \Cs_v^{\theta}|\right)_{v \in \Z^d}$, for each color $s$, where $\Cs_v^{\theta}$ is the open cluster of $v$ in Bernoulli site percolation of parameter $\theta$. We denote by $\textbf{P}_{\theta}^{\text{site}}$ the corresponding measure. Thus
$$\textbf{P} \big( \exists r \in \Gamma (0,\cdot),  \ |r| = mk, \ T(r) \leq m \big) \leq (S+1) \textbf{P}_{\theta}^{\text{site}} \left(  \sup_{|r|=n} \frac{1}{|r|} \sum_{v \in r}  | \Cs_v | \geq \frac{1}{\alpha(S+1)} \right).$$

Let us note that the right term depends only on $\theta$ and $S$, not on $p$. To finish, we shall prove the following lemma.

\begin{lem}\label{aa}
 Let $\theta$ be a real number in $(0, p_c)$ and $S$ be an integer larger than $5$.
There exist positive real numbers $N(\theta, S), K(\theta, S), C_1(\theta, S), C_2(\theta, S)$ such that, for each integers $n \geq N$ and $k \geq K$,
\begin{equation}
 \textbf{P}_{\theta}^{\text{site}} \left(  \sup_{|r|=n} \frac{1}{|r|} \sum_{v \in r}  | \C_v| \geq \frac{1}{\alpha(S+1)} \right) \leq C_1 e^{-C_2 (nk)^{\frac{1}{5}}}, \ \text{with } \alpha:=\frac{1}{k} + \frac{1}{n}.
\end{equation}
\end{lem}
The idea is to use greedy lattice animals to study this probability.

\begin{proof}[Proof of Lemma~\ref{aa}]
Fix $0 < \theta < p_c$ and $S \geq 5$.

To prove Lemma~\ref{aa} by means of Proposition~\ref{martthm}, we need to relate the dependent family $|\Cs_v|$ to some independent variables. For that, we repeat arguments of Fontes and Newman~\cite{fetn}. We consider an i.i.d. family $\{ \tilde{Y}\}_{ v \in \Z^d}$ of random lattice animals, which are equidistributed with $\Cs_0$, the occupied cluster of this origin in Bernoulli site percolation. We put $\tilde{\Cs}_v := v + \tilde{Y}_v$ and $\tilde{U}_v := \sup \left\{| \tilde{\Cs}_u |; u \in \Z^d, \ v \in \tilde{\Cs}_u \right \}$, where the sup of an empty set is taken to be zero. 

Then we have the following stochastic domination inequality $ \left\{ | \Cs_v | \right\}_{v \in \Z^d} \prec \left\{ | \tilde{U}_v | \right\}_{v \in \Z^d}$ and for any self-avoiding path $r$, we have
\begin{align*}
\frac{1}{|r|} \sum_{v \in r} | \tilde{U}_v |  & \leq 2 \sup_{\xi' \supset r} \frac{1}{|\xi'|} \sum_{v \in \xi'} | \tilde{\Cs}_v |^2\\
\intertext{where the sup is over all lattice animals (\textit{i.e.} connected subset of $\Z^d$) containing $r$. Then, by combining this with the stochastic domination, we have}
\textbf{P}_{\theta}^{\text{site}} \left(  \sup_{|r|=n} \frac{1}{|r|} \sum_{v \in r}  | \Cs_v| \geq \frac{1}{\alpha(S+1)} \right) & \leq \textbf{P}_{\theta}^{\text{site}} \left( \sup_{|\xi| \geq n} \frac{1}{|\xi|} \sum_{v \in \xi} | \tilde{\Cs}_v |^2 \geq \frac{1}{2 \alpha(S+1)} \right) \\
& \leq \sum_{k=n}^{+ \infty} \textbf{P}_{\theta}^{\text{site}} \left( \sup_{|\xi| = k} \frac{1}{|\xi|} \sum_{v \in \xi} | \tilde{\Cs}_v |^2 \geq \frac{1}{2 \alpha(S+1)} \right).
\end{align*}

We put $W:= \lim_{n \rightarrow + \infty} \frac{1}{n} \sup_{| \xi|=n} \sum_{v \in \xi} | \Cs_v|^2$.

To use Proposition~\ref{martthm}, we need to have
\begin{equation}\label{need1}
\frac{1}{2 \alpha (S+1)} \geq W+1,
\end{equation}
and
\begin{equation}\label{need2}
\text{the random variables must be bounded by above.}
\end{equation}
The point~\eqref{need1} is equivalent to $\alpha= \frac{1}{k}+ \frac{1}{n} \leq \frac{1}{2(W+1)(S+1)}$. Thus we take two integers $N$ and $K>0$ such that if $n \geq N$ and $k \geq K$, $\frac{1}{k}+ \frac{1}{n} \leq \frac{1}{2(W+1)(S+1)}$.
For the point~\eqref{need2}, we truncate the random variables $|\tilde{\Cs}_v|^ 2$ up to an appropriate number and we control the possibility of a too large cluster.

Indeed, in sub-critical percolation, the tail of the distribution of $|\Cs_0|$ decays exponentially (see Kesten~\cite{cluster}). For each real number $\theta$ in $(0,p_c)$, there exists a constant $C(\theta)>0$, such that
$$\textbf{P}_{\theta}^{\text{site}} \left( | \Cs_0 | \geq n \right) \leq e^{-C n}, \ \forall n \in \N^*.$$

Thus
\begin{align*}
\textbf{P}_{\theta}^{\text{site}} & \left( \sup_{|\xi| = j}  \frac{1}{|\xi|} \sum_{v \in \xi} | \tilde{\Cs}_v |^2  \geq \frac{1}{2 \alpha(S+1)} \right)\\
&  \leq  \textbf{P}_{\theta}^{\text{site}} \left( \sup_{|\xi| = j} \frac{1}{|\xi|} \sum_{v \in \xi} |\tilde{\Cs}_v |^2 \geq W+ 1 \right) \\
&  =  \textbf{P}_{\theta}^{\text{site}} \left( \sup_{|\xi| = j} \frac{1}{|\xi|} \sum_{v \in \xi} |\tilde{\Cs}_v |^2 \geq W+ 1 \text{ and } \forall v \in [-j,j]^d, \ |\tilde{\Cs}_v |^2 < j^{\frac{2}{5}} \right) \\
& \phantom{= \textbf{P}_{\theta}^{\text{site}}} + \textbf{P}_{\theta}^{\text{site}} \left( \sup_{|\xi| = j} \frac{1}{|\xi|} \sum_{v \in \xi} |\tilde{\Cs}_v |^2 \geq W+ 1 \text{ and } \exists v \in [-j,j]^d, \ |\tilde{\Cs}_v |^2 \geq j^{\frac{2}{5}} \right) \\
& \leq \textbf{P}_{\theta}^{\text{site}} \left( \sup_{|\xi| = j} \frac{1}{|\xi|} \sum_{v \in \xi} \min \left(|\tilde{\Cs}_v |^2, j^{\frac{2}{5}} \right) \geq W+ 1 \right) + \textbf{P} \left( \exists v \in [-j,j]^d, \ |\tilde{\Cs}_v |^2 \geq j^{\frac{2}{5}} \right). 
\end{align*}

By Proposition~\ref{martthm}, for $j$ large enough, we have
\begin{align*}
\textbf{P}_{\theta}^{\text{site}} \left( \sup_{|\xi| = j} \frac{1}{|\xi|} \sum_{v \in \xi} | \tilde{\Cs}_v |^2 \geq \frac{1}{2 \alpha(S+1)} \right) & \leq C_1 \exp \left\{ - \frac{C_2 j}{\left(j^{\frac{2}{5}}\right)^2} \right\} + \sum_{v \in [-j,j]^d} \textbf{P}_{\theta}^{\text{site}} \left( |\tilde{\Cs}_v |^2 \geq j^{\frac{2}{5}} \right) \\
& \leq C_1 e^{ - C_2 j^{\frac{1}{5}}} + 2dj (2j+1)^d \textbf{P}_{\theta}^{\text{site}} \left( |\tilde{\Cs}_0 | \geq j^{\frac{1}{5}} \right)\\
& \leq C_3 e^{-C_4 j^{\frac{1}{5}}}, \ \text{for } j \text{ large enough}.
\end{align*}

Therefore, there exist positive real numbers $K,N,C_1,C_2$ such that, for all integers $n >N$ and $k> K$,
\begin{equation*}
 \textbf{P}_{\theta}^{\text{site}} \left( \sup_{|\xi| = j} \frac{1}{|\xi|} \sum_{v \in \xi} |\tilde{\Cs}_v |^2 \geq \frac{1}{2 \alpha(S+1)} \right) \leq C_1 e^{-C_2 j^{\frac{1}{5}}}, \text{ for all } j \geq n.
\end{equation*}

Thus, there exist positive real numbers $K, N, C_1, C_2$ such that, for all integers $n >N$ and $k > K$,
\begin{equation*}
 \textbf{P}_{\theta}^{\text{site}} \left( \sup_{|r| = n} \frac{1}{|r|} \sum_{v \in r} |\Cs_v | \geq \frac{1}{2 \alpha(S+1)} \right) \leq C_1 e^{-C_2 n^{\frac{1}{5}}}, \text{ with } \alpha= \frac{1}{k}+ \frac{1}{n}.
\end{equation*}

\end{proof}

This concludes the proof of Lemma~\ref{petitc}.

\section*{Acknowledgement} The author would like to thank Olivier Garet for his guidance and helpful discussion during the achievement of this work.


\begin{thebibliography}{16}
\bibitem[Bil95]{billingsley} Patrick Billingsley. \textit{Probability and Measure}. Wiley-Interscience, 1995.
\bibitem[Boi90]{boivin} Daniel Boivin. First passage percolation: the stationary case. \textit{Probability Theory and Related Fields}, 86(4):491--499, 1990.
\bibitem[CGGK93]{gla1} J. Theodore Cox, Alberto Gandolfi, Philip S. Griffin and Harry Kesten. Greedy lattice animals. {I}. {U}pper bounds. \textit{The Annals of Applied Probability}, 3(4):1151--1169, 1993.
\bibitem[CK81]{coxkesten} J. Theodore Cox and Harry Kesten. On the continuity of the time constant of first-passage percolation. \textit{Journal of Applied Probability}, 18:809--819, 1981.
\bibitem[Cox80]{cox} J. Theodore Cox. The time constant of first-passage percolation on the square lattice. \textit{Advances in Applied Probability}, 12(4):864--879,1980.
\bibitem[FvN93]{fetn} Luiz Fontes and Charles M. van Newman. First passage percolation for random colorings of $\mathbb{Z}^2$. \textit{The Annals of Applied Probability}, 3(3):746--762, 1993. 
\bibitem[GK94]{gla2}Alberto Gandolfi and Harry Kesten. Greedy lattice animals. {II}. {L}inear growth. \textit{The Annals of Applied Probability}, 4(1):76--107, 1994.
\bibitem[Gri99]{grimmett-book} Geoffrey Grimmet. \textit{Percolation}, volume 321 of \textit{Grundlehren der Mathematischen Wissenschaften [Fundamental Principle of Mathematical Sciences]}. Springer-Verlag, Berlin, second edition, 1999. 
\bibitem[How04]{howard} C. Douglas Howard. Models of first-passage percolation. In  \textit{Probability on discrete structures}, volume 110 of \textit{Encylopaedia of Mathematical Sciences}, 125--173. Springer, Berlin, 2004
\bibitem[HW65]{HW} J.M. Hammersley and D. J. A. Welsh. First-passage percolation, subadditive processes, stochastic networks, and generalized renewal theory. In \textit{Proc. Internat. Res. Semin., Statist. Lab., Univ. California, Berkeley, Calif.}, 61--110. Springer-Verlag, New York, 1965.
\bibitem[Kes81]{cluster} Harry Kesten. Analyticity properties and power law estimates of functions in percolation theory. \textit{Journal of Statistical Physics},25(4):717--756,1981.
\bibitem[Kes86]{kesten} Harry Kesten. Aspects of first passage percolation. In \textit{\'Ecole d'\'et\'e de probabilit\'es de Saint-Flour, XIV---1984}, volume 1180 of \textit{Lecture Notes in Mathematics}, 125--264. Springer, Berlin, 1986.
\bibitem[Kes03]{surveyK} Harry Kesten. First-passage percolation. In \textit{From classical to modern probability}, volume 54 of \textit{Progr. Probab.}, 93--143. Birkh\"auser, Basel, 2003.
\bibitem[Kin68]{MR0254907} J. F. C. Kingman. The ergodic theory of subadditive stochastic processes. \textit{Journal of the Royal Statistical Society. Series B. Methodological}. 30:499--510, 1968.
\bibitem[Lig85]{MR806224} Thomas M. Liggett. An improved subadditive ergodic theorem. \textit{The Annals of Probability}, 13(4):1279--1285, 1985.
\bibitem[Mar02]{Martin} James B. Martin. Linear growth for greedy lattice animals. \textit{Stochastic Processes and their Applications}, 98(1):43--66, 2002.
\end{thebibliography}

\end{document}